\documentclass[a4paper,11pt]{article}

\usepackage[utf8]{inputenc}
\usepackage[english]{babel}
\usepackage{amsmath}
\usepackage{amsthm}
\usepackage{amsfonts}
\usepackage{amssymb}
\usepackage{caption}
\usepackage{hyperref} 
\usepackage[english]{cleveref} 
\usepackage{dsfont}
\usepackage{emptypage}
\usepackage[shortlabels]{enumitem}
\usepackage{extpfeil}
\usepackage{faktor}
\usepackage{fancyhdr}
\usepackage{float}
\usepackage{graphicx}
\usepackage{indentfirst}
\usepackage{mathdots}
\usepackage{mathtools}
\usepackage{mathrsfs}
\usepackage{multirow}
\usepackage{rotating}
\usepackage{stackrel}
\usepackage{stmaryrd}
\usepackage{subcaption}
\usepackage{tikz}
\usepackage{todonotes}
\usepackage{verbatim}

\usetikzlibrary{patterns,arrows,calc,matrix,decorations.pathreplacing}

\newtheorem{thm}{Theorem}[section]

\newtheorem{prop}[thm]{Proposition}

\newtheorem{lemma}[thm]{Lemma}
\theoremstyle{definition}
\newtheorem{definition}[thm]{Definition}
\newtheorem{example}[thm]{Example}

\theoremstyle{remark}
\newtheorem{remark}[thm]{Remark}

\DeclareMathOperator{\rk}{rk}

\def\N{\mathbb{N}}
\def\Z{\mathbb{Z}}
\def\Q{\mathbb{Q}}

\def\ux{\underline{x}_i}
\def\uy{\underline{y}^i}
\def\uf{\underline{f}}
\def\uz{\underline{z}}
\newcommand{\defeq}{\stackrel{\textnormal{def}}{=}}

\DeclareMathOperator{\sgn}{sgn}

\begin{document}
\title{Orientable arithmetic matroids}
\author{Roberto Pagaria}
\maketitle

\begin{abstract}
The theory of matroids has been generalized to oriented matroids and, recently, to arithmetic matroids.
We want to give a definition of ``oriented arithmetic matroid'' and prove some properties like the ``uniqueness of orientation''.
\end{abstract}


The aim of this paper is to relate two different generalizations of matroids: the oriented matroids and the arithmetic matroids.

Oriented matroids have a large use in mathematics and science;
they are related to the simplex method for linear programming, to the chirality of molecules in theoretical chemistry, and to knot theory.
For instance, the Jones polynomial of a link is a specialization of the signed Tutte polynomial (see \cite{Kauffman}) of an oriented graphic matroid \cite{Thistlethwaite, Jaeger}.
Another interesting fact is the correspondence between oriented matroids and arrangements of pseudospheres \cite{Folkman-Lawrence} that generalizes the correspondence between realizable matroids and central hyperplane arrangements.

Arithmetic matroids appear as the combinatorial object for the cohomology module of the complement of a toric arrangement \cite{DeConciniProcesiToricArr,Moci,CDDMP}.
The study of toric arrangements is related to zonotopes, partition functions, box splines, and Dahmen-Micchelli spaces (see \cite{DeConciniProcesiVergne,DeConciniProcesiBook,Moci}).
The obvious correspondence between realizable arithmetic matroids and central toric arrangements has not been generalized to the non-realizable cases, so far.

With the aim of filling this gap, we define a class of well-behaviour arithmetic matroids which we call \textit{orientable arithmetic matroids} (see \Cref{def:oam}) 
hoping that these correspond to ``arrangements of pseudo-tori''.

\medskip
An $r \times n$ matrix with integer coefficients describes at the same time a central toric arrangement, an oriented matroid, and an arithmetic matroid.
It comes natural to say that two matrices are equivalent if they describe the same toric arrangement.
Geometrically, the group $\operatorname{GL}_r(\Z)\times (\Z/2\Z)^n$ acts on the space $\operatorname{M}(r,n;\Z)$ by left multiplication and sign reverse of the columns.
Two realizations (i.e. matrices) of the arithmetic matroid are equivalent if and only if they belong to the same $\operatorname{GL}_r(\Z)\times (\Z/2\Z)^n$-orbit.

The space $\operatorname{M}(r,n;\Z)$ is included in $\operatorname{M}(r,n;\Q)$ and the action of $\operatorname{GL}_r (\Z)\times (\Z/2\Z)^n$ extends naturally to the one of $\operatorname{GL}_r (\Q)\times (\Z/2\Z)^n$.
Theorem 3.12 in \cite{Pagaria17} shows that all representations of an arithmetic matroid belong to the same $\operatorname{GL}_r (\Q)\times (\Z/2\Z)^n$-orbit.
By this fact, it can be easily deduced that representable arithmetic matroids have a unique orientation.
We extend this result to the non-representable case, showing (\Cref{thm:uniqueness_orientation}) that orientable arithmetic matroids have an unique orientation (up to re-orientation).

We start by recalling some standard definitions and giving the compatibility condition, eq.~\eqref{cond:GP},  between the orientation and the multiplicity function of an oriented arithmetic matroid.
The condition \eqref{cond:GP} coincides with the Pl\" uker relation for the Grassmannian.
We prove that oriented arithmetic matroids are closed under deletion, contraction, and duality.
Next, we show that the condition \eqref{cond:GP} implies 
a generalization of the Leibniz rule for the determinant.
We state and prove a result about the uniqueness of orientation so that it makes sense to speak of orientable arithmetic matroids instead of oriented arithmetic matroids.
Finally, we show that orientable arithmetic matroids, upon forgetting the multiplicity function, are realizable matroids (see \Cref{prop:strongGCD+orient_realiz}).
Moreover, we state the condition ``strong GCD'' (see \Cref{def:strongGCD}) implying the realizability of orientable arithmetic matroids. 

All the discussion can be generalize to quasi-arithmetic matroids and so to matroids over $\Z$ (see \cite{FM16})

\section{Definitions}
Let $E$ be a finite totally ordered set.
We will frequently make use of $r$-tuples of elements of $E$, so with an abuse of notation for any $A= \{ a_1, \dots, a_r\} \subset E$ we will write $A$ for the increasing tuple $(a_1, \dots, a_r)$.

We give the definition of a matroid in terms of its basis, since
\cite[Theorem 1.2.3]{Oxley} shows that it is equivalent to the one given in terms of independent sets.
\begin{definition}
A matroid over a finite set $E$ is a non-empty set $\mathcal{B}\subset \mathcal{P}(E)$ such that
\begin{equation} \label{eq:exchange_property}
\forall \, B_1, B_2 \in \mathcal{B}\, \forall \, x \in B_1 \setminus B_2 \, \exists \, y \in B_2 \setminus B_1 \textnormal{ such that } B_1 \setminus \{x\} \cup \{y \} \in \mathcal{B}.
\end{equation}
\end{definition}
Since this definition of matroids is cryptomorphic to the one involving the \textit{rank function} (see \cite[Theorem 1.3.2]{Oxley}), we denote a matroid with the pair $(E,\rk)$.

Throughout this paper we will denote the $r$-tuples $(y_i,x_2, \dots, x_r)$ by $\underline{x}_i$ and $(y_0, \dots, y_{i-1}, y_{i+1}, \dots, y_r)$ by $\underline{y}^i$, for $i \in \{0, \dots, r\}$, where $\underline{x}=(x_2, \dots, x_r)$ and $\underline{y}=(y_0, \dots, y_r)$.

\begin{definition}[{\cite[Definition 3.5.3]{OrientedMatroidsBook}}] \label{def:chirotope}
A \textit{chirotope} is a function $\chi: E^r \rightarrow \{-1,0,1\}$ such that:
\begin{enumerate}[(B1)]
\setcounter{enumi}{-1}
\item it is not identically zero, i.e. $\chi \not \equiv 0$,
\item it is alternating, i.e. $\chi(\sigma \underline{x})= \sgn (\sigma) \chi(\underline{x})$ for all $\sigma \in \mathfrak{S}_n$,
\item \label{cond:B2} for all $x_2, \dots, x_r$ and all $y_0, \dots, y_{r} \in E$ such that 
\[ \chi(\underline{x}_i) \chi(\underline{y}^i)\geq 0, \]
for all $i>0$, then we have
\[ \chi(\underline{x}_0) \chi(\underline{y}^0)\geq 0. \]
\end{enumerate}
\end{definition}

\begin{definition}[{\cite[p.~134]{OrientedMatroidsBook}}]
The \textit{re-orientation} with respect to $A \subseteq E$ of a chirotope $\chi$ is the chirotope $\chi'$ defined by
\[ \chi'(\underline{x})=(-1)^{|A\cap \{x_1, \dots,x_r\}|}\chi(\underline{x}) .\]
Two chirotopes are equivalent if one is a re-orientation of the other one.
\end{definition}

The set $\{\{b_1, \dots, b_r\} \subset E \mid \chi(b_1, \dots, b_r) \neq 0\}$ is the matroid over $E$ associated with the chirotope $\chi$.
A well-known cryptomorphism of Lawrence \cite{Lawrence82} between \textit{oriented matroids} and chirotopes is stated in \cite[Theorem 3.5.5]{OrientedMatroidsBook}.

For every matroid $\mathcal{M}=(E,\rk)$ and every subset $A \subseteq E$ we denote by $\mathcal{M}/A$ the \textit{contraction} of $A$ and  with $\mathcal{M} \setminus A$ the \textit{deletion} of $A$.

Let us recall the definition of ``arithmetic matroid'' introduced in \cite{DM13,BM14}.

\begin{definition}
A \textit{molecule} $(A,B)$ of the matroid  is a pair of sets $A\subset B \subseteq E$ such that the matroid $(\mathcal{M}/A)\setminus B^c$ has a unique basis.
\end{definition}

For $A\subseteq E$ we denote the maximal subset $S\supseteq A$ of rank equal to $\rk(A)$ by $\overline{A}$.

\begin{definition}
An \textit{arithmetic matroid} is $(E,\rk,m)$ such that $(E,\rk)$ is a matroid and $m: \mathcal{P}(E) \rightarrow \N_+=\{ 1,2,\dots\}$ a function satisfying:
\begin{enumerate}
\item if $A \subseteq E$ and $x \in E$ is dependent on $A$, then $m(A \cup \{v\}) | m(A)$;
\item if $A \subseteq E$ and $x \in E$ is independent on $A$, then $m(A)| m(A \cup \{v\}) $;
\item \label{eq:cond_on_molecule} if $(A,B)$ is a molecule then 
\[ m(A)m(B)=m(B \cap \overline{A})m((B\setminus \overline{A})\cup A);\]
\item if $(A,B)$ is a molecule then 
\[ \rho (A,B)\defeq \sum_{A \subseteq S \subseteq B} (-1)^{|\overline{A}\cap B|-|S|} m(S) \geq 0.\]
\end{enumerate}
We call $m$ the \textit{multiplicity function}.
\end{definition}

\begin{definition}\label{def:oam}
An \textit{oriented arithmetic matroid} $(E,\rk,m,\chi)$ is a matroid 
$(E,\rk)$ of rank $r$ together with two structures: a chirotope $\chi:E^r \rightarrow \{-1,0,1\}$ and a multiplicity function $m: \mathcal{P}(E) \rightarrow \N_+$ such that:
\begin{enumerate}
\item The unoriented matroid associated with the chirotope $\chi$ is the matroid $(E,\rk)$.
\item The triple $(E,\rk, m)$ is an arithmetic matroid.
\item For all $x_2, \dots, x_r$ and all $y_0, \dots, y_{r} \in E$ the following equality holds
\begin{equation}\label{cond:GP}
\sum_{i=0}^r (-1)^i \chi(\underline{x}_i)m(\underline{x}_i) \chi(\underline{y}^i)m(\underline{y}^i)=0, \tag{GP}
\end{equation}
where $\underline{x}_i = (y_i,x_2 \dots ,x_r)$ and $\underline{y}^i = (y_0, \dots y_{i-1},y_{i+1} \dots y_r)$.
\end{enumerate}
\end{definition}
Since the rank function $\rk$ is completely determined by the chirotope $\chi$, we omit $\rk$ and write $(E,\chi,m)$ for a oriented arithmetic matroid.

\begin{remark}
Our property \eqref{cond:GP}, related to the Grassmannian-Pl\" uker relations, implies the properties $(\textnormal{GP}_r)$, for all $r$,  defined in \cite[Definition 10.3]{Lenz2}.
\end{remark}

Notice that the compatibility condition \eqref{cond:GP} involves only the values of the multiplicity function on the basis of $(E,\rk)$.
\begin{remark}
The condition \eqref{cond:GP} implies \ref{cond:B2} of \Cref{def:chirotope}.
\end{remark}


\section{Deletion}
The deletion of $A\subset E$ is an operation defined for matroids \cite[p.~22]{Oxley}, for oriented matroids \cite[p.~133]{OrientedMatroidsBook}, and for arithmetic matroids \cite[section 4.3]{DM13} \cite[section 3]{BM14}.
We now define a deletion operation for oriented arithmetic matroids.

The triple $(E\setminus A, \chi \setminus A, m \setminus A)$ satisfies the first two conditions of \Cref{def:oam}.
\begin{prop}
The triple $(E\setminus A, \chi \setminus A, m \setminus A)$ is an oriented arithmetic matroid.
\end{prop}
\begin{proof}
Let $s$ be the rank of $E\setminus A$ and $\underline{f}=(a_1, \dots, a_{r-s}) \subseteq A$ such that $\rk((E\setminus A) \cup \underline{f})=r$.
Consider the elements $x_2, \dots, x_s$ and $y_0, \dots, y_s$ in $E\setminus A$.
For all $0\leq i \leq s$, the triples $(\ux, \uy, \uf)$ and $(\uy, \ux, \uf)$ are molecules.
The equality
\[ m(\ux \cup \uy)^2 m(\ux \cup \uf) m(\uy \cup \uf)=m(\ux \cup \uy \cup \uf)^2 m(\ux) m(\uy)\]
follows from condition \eqref{eq:cond_on_molecule} applies to the two molecules.
Notice that $\ux \cup \uy$ does not depend on $i$ so we can denote it $\underline{x} \cup \underline{y}$.
We have
\begin{multline*}
m(\underline{x} \cup \underline{y} \cup \uf)^2 \sum_{i=0}^s (-1)^i \chi(\ux \cup \uf) m(\ux) \chi(\uy \cup \uf) m(\uy) =\\
= m(\underline{x} \cup \underline{y})^2 \sum_{i=0}^s (-1)^i \chi(\ux \cup \uf) m(\ux \cup \uf) \chi(\uy \cup \uf) m(\uy \cup \uf).
\end{multline*}
The right side is, up to a non-zero scalar, the equation \eqref{cond:GP} applied to $x_2, \dots, x_s, a_1, \dots, a_{r-s}$ and $y_0, \dots, y_s, a_1, \dots, a_{r-s}$ for the oriented arithmetic matroid $(E, \chi, m)$.
Therefore, we have proven the claimed equality
\[\sum_{i=0}^s \chi(\ux \cup \uf) m(\ux) \chi(\uy \cup \uf) m(\uy)=0. \qedhere\]
\end{proof}

\section{Contraction}
The contraction (or restriction) of $A\subset E$ is an operation defined for matroids \cite[p.~22]{Oxley}, for oriented matroids \cite[p.~134]{OrientedMatroidsBook}, and for arithmetic matroids \cite[section 4.3]{DM13} \cite[section 3]{BM14}.
We now define a contraction operation for oriented arithmetic matroids.

Let $A$ be a subset of $E$ and call $r-s$ its rank.
We choose an independent list $\uf=(a_1, \dots, a_{r-s})$ of elements in $A$.
Define $\chi/A:(E\setminus A)^s \rightarrow \{-1,0,1\}$ as $\chi/A (\uz)=\chi(\uz \cup \uf)$ and $m/A(S)=m(A\cup S)$.
\begin{prop}
The triple $(E\setminus A, \chi/A, m/A)$ is an oriented arithmetic matroid.
\end{prop}
\begin{proof}
We call $T=A\setminus \uf$ and fix the elements $x_2 \dots x_s$ and $y_0 \dots y_s$ of $E\setminus A$.
Observe that $(\uf, T, \ux)$ and  $(\uf, T, \uy)$ are molecules of $(E,\rk)$.
Thus
\[ m(A)^2 m(\ux \cup \uf)m(\uy \cup \uf)= m(\uf)^2 m(\ux \cup A)m(\uy \cup \uf).\]
Since $m(A)$ and $m(f)$ are nonzero, then condition \eqref{cond:GP} for $\underline{x}$ and $\underline{y}$ in the contracted matroid is equivalent to condition \eqref{cond:GP} for $\underline{x}\cup \uf$ and $\underline{y}\cup \uf$ in the original matroid.
\end{proof}

\section{Duality}
The duality is an operation defined for matroids \cite[chapter 2]{Oxley}, for oriented matroids \cite[p.~135]{OrientedMatroidsBook}, and for arithmetic matroids \cite[p.~339]{DM13} \cite[p.~5526]{BM14}.
We now define duality for oriented arithmetic matroids.

Recall that the set $E$ is ordered.
For every $\uz=(z_1, \dots, z_{k})\subseteq E$ we call $\uz'$ the complement of $\uz$ in $E$ with some arbitrary order and let $\sigma(\uz, \uz')$ be the sign of the permutation that reorders the list $(\uz, \uz')$ as they appear in $E$.
We define $\chi^*:E^{n-r} \rightarrow \{-1,0,1\}$ as
\[ \chi^*(\uz)=\chi(\uz') \sigma(\uz,\uz')\]
and the multiplicity function $m^*: \mathcal{P}(E) \rightarrow \N_+$ as $m^*(\uz)=m(\uz')$.
\begin{prop}
The triple $(E,\chi^*, m^*)$ is an oriented arithmetic matroid.
\end{prop}
\begin{proof}
Let $\underline{x}=(x_2, \dots, x_{n-r})$ and $\underline{y}=(y_0, \dots, y_{n-r})$ be two sublists of $E$.
Coherently with the notation above, let $\underline{x}'=(x_0', \dots, x_r')$ and $\underline{y}'=(y_2', \dots, y'_r)$ be their complements.
For every $0\leq i \leq n-r$ the element $y_i$ is equal to $x_k$ or $x_j'$.
In the first case $\chi^*(\ux)=0$ and in the second case 
\[
\chi^*(\ux)=\chi(\underline{x}^{\prime j})\sigma(\ux,\underline{x}^{\prime j}) = (-1)^{n-r+1+j}\chi(\underline{x}^{\prime j})\sigma(\underline{x},\underline{x}').
\]
Analogously, if $y_i=x_j'$ then
\[ \chi^*(\uy)=\chi(\underline{y}'_j)\sigma(\uy,\underline{y}'_j)= (-1)^{n-r+i}\chi(\underline{y}'_j)\sigma(\underline{y},\underline{y}')\]
where $\underline{y}'_j=(x'_j, y'_2, \dots, y_r')$.
If $y_i=x'_j$, then 
$m^*(\ux)=m(\underline{x}^{\prime j})$ and $m^*(\uy)=m(\underline{y}'_j)$.
Thus, up to a sign, the condition \eqref{cond:GP} for $\underline{y}'$ and $\underline{x}'$ in the original matroid implies condition \eqref{cond:GP} for $\underline{x}$ and $\underline{y}$ in the dual matroid.
\end{proof}

\section{GP-functions}
We now study functions satisfying a relation that looks like the Pl\" uker relation for the Grassmannian.
A posteriori all these functions are nothing else that the determinant $\det: V^r \rightarrow \Q$ restricted to a finite (multi-)set $E\subset V$.

\begin{definition}
A map $f:E^r \rightarrow \Q$ is a \textit{GP-function} if it is alternating and for all $\underline{x}\in E^{r-1}$ and all  $\underline{y}\in E^{r+1}$ the following equality holds
\[\sum_{i=0}^r (-1)^i f(y_i,x_2 \dots ,x_r) f(y_0, \dots y_{i-1},y_{i+1} \dots y_r)=0 \]
\end{definition}

The main examples of GP-function are the function $\chi m$ for every oriented arithmetic matroid.
Another example is given by a map $i:E\rightarrow V$, where $V$ is a $\Q$-vector space of dimension $n$, and consider the function $\det:V^n \rightarrow \Q$.
The determinant is a GP-function and the following theorem is a generalization of the Leibniz formula.

\begin{thm}\label{thm:Leibniz_formula}
Let $f:E^r \rightarrow \mathbb{Q}$ be a GP-function.
Then for all $(a_1,\dots, a_r)$ in $E^{r}$ and $(b_1, \dots, b_r) \in E^{r}$ the following formula holds:
\begin{equation}\label{eq:Leibniz_formula}
\sum_{\sigma \in \mathfrak{S}_r}(-1)^{\operatorname{sgn} \sigma} \prod_{i=1}^r f(a_1, \dots, b_{\sigma(i)}, \dots, a_r)=f(a_1, \dots, a_r)^{r-1} f(b_1, \dots, b_r),
\end{equation}
where $b_{\sigma(i)}$ substitutes $a_i$.
\end{thm}
\begin{proof}
We prove lemma by induction, the base case $r=2$ is trivial.
We fix $(a_1,\dots, a_r)\in E^{r}$ and $(b_1, \dots, b_r) \in E^{r}$.
Let $g:E^{r-1} \rightarrow \Q$ be the GP-function defined by
\[ g(x_2, \dots, x_r)= f(a_1,x_2, \dots, x_r).\]
By inductive step we have
\begin{equation}\label{eq:inductive_step}
\sum_{\sigma \in \mathfrak{S}_{r-1}}(-1)^{\operatorname{sgn} \sigma} \prod_{i=2}^r g(a_2, \dots, c_{\sigma(i)}, \dots, a_r)= g(a_2, \dots, a_r)^{r-2} g(c_2, \dots, c_r).
\end{equation}
The left hand side of the eq.~\eqref{eq:Leibniz_formula} can be rewritten as:
\begin{equation}\label{eq:uff_ci_siamo_quasi}
\sum_{j=1}^r f(b_{j},a_2, \dots, a_r) \sum_{\sigma \in \mathfrak{S}_{r-1}} (-1)^{\sgn \sigma + \sgn \tau_j} \prod_{i=2}^r f(a_1, \dots, b_{\sigma(\tau_j(i))}, \dots, a_r),
\end{equation}
where $\tau_j=(1,j)$ and $\mathfrak{S}_{r-1}$ is the subgroup of $\mathfrak{S}_r$ of permutations that fix the element $1$.
Now, for every $j$, we use eq.~\eqref{eq:inductive_step} with $c_i= b_{ \tau_j(i)}$ to manipulate expression \eqref{eq:uff_ci_siamo_quasi}:
\[
\begin{split}
f(a_1, \dots, a_r)^{n-2} \Bigl[ f(b_1,a_2, \dots, a_r) f(a_1,b_2, \dots, b_r) - &\sum_{j=1}^r  f(b_{j},a_2, \dots, a_r) \cdot \\
&f(a_1,b_2, \dots, b_1, \dots, b_r)\Bigl] 
\end{split}
\]
that it is equal to left hand side of \eqref{eq:Leibniz_formula} since $f$ is a GP-function.
\end{proof}

\begin{lemma}\label{lemma:BG1_implies_BG}
Let $f$ and $g$ be two GP-functions and $B \in  E^r$.
Suppose that $f(B)=g(B) \neq 0$ and $f(C)=g(C)$ for all $C\in E^r$ such that $|\{i \mid c_i \neq b_i\}|=1$, then $f=g$.
\end{lemma}
\begin{proof}
We use \Cref{thm:Leibniz_formula} for the function $f$ and $g$.
We set $\{a_1,\dots, a_r\}=B$ in eq.~\eqref{eq:Leibniz_formula}, the left hand side for $f$ and $g$ are equal, so 
\[f(a_1, \dots, a_r)^{r-1} f(b_1, \dots, b_r)=g(a_1, \dots, a_r)^{r-1} g(b_1, \dots, b_r).\]
By hypothesis $f(a_1, \dots, a_r)=g(a_1, \dots, a_r)\neq 0$, thus we have $f(b_1, \dots, b_r)=g(b_1, \dots, b_r)$ for all $b_i$, $i=1,\dots, r$.
\end{proof}

\section{Uniqueness of the orientation} \label{sect:uniq_orientation}

\begin{thm}\label{thm:uniqueness_orientation}
Let $(E,\chi, m)$ and $(E;\chi',m)$ be two oriented arithmetic matroids over the same matroid $(E,\rk)$.
Then $\chi'$ is a re-orientation of $\chi$.
\end{thm}
We fix a total order on $E\simeq [n]$ such that $[r]$, the first $r$ elements, are a basis of the matroid.

The basis graph of a matroid is first studied in \cite{MaurerI} and \cite{MaurerII}.
\begin{definition}
The \textit{basis graph} $\mathcal{BG}$ of a matroid $(E, \mathcal{B})$ is the graph on the set $\mathcal{B}$ of vertexes  with an edge between two vertexes $B_1$ and $B_2$ if $|B_1 \setminus B_2|=1$.
\end{definition}

Once chosen a basis $B_0$ of a matroid, we define $\mathcal{BG}_1$ to be the induced subgraph of $\mathcal{BG}$ whose vertexes are all vertexes adjacent to $B_0$.
Define $\mathcal{BG}_{\leq 1}$ the induced subgraph whose vertexes are the ones adjacent to $B_0$ and $B_0$ itself. 

Suppose that $\chi([r])=\chi'([r])$,
\Cref{lemma:dist_1} proves that, up to reorientation, $\chi$ and $\chi'$ coincides on all vertexes of distance one from $[r]$.
\Cref{lemma:induction_basis_graph} proves that $\chi(B)=\chi'(B)$ using \Cref{thm:Leibniz_formula}.

\begin{definition}
Let $\mathcal{G}$ be the bipartite graph on vertexes $E$ and an edge between $i\in B_0$ and $j \in E\setminus B_0$ if $B_0\setminus \{i\} \cup \{j\}$ is a basis.
We call this graph the $B_0$-\textit{fundamental circuit graph}.
\end{definition}

\begin{definition}
The \textit{Line graph} $L(\mathcal{G})$ of a graph $\mathcal{G}=(\mathcal{V}, \mathcal{E})$ is the graph whose set of vertexes is the set $\mathcal{E}$ of edges in $\mathcal{G}$.
The graph $L(\mathcal{G})$ has an edge between $e_1$ and $e_2\in \mathcal{E}$ if and only if the edges $e_1$ and $e_2$ are incident in $\mathcal{G}$.
\end{definition}

The Line graph of $\mathcal{G}$ is the graph $\mathcal{BG}_1$.
A \textit{coordinatizing path} in $\mathcal{G}$ is a spanning forest of the graph $\mathcal{G}$.
We choose a coordinatizing path $P$ of the graph $\mathcal{G}$ and its Line graph $L(P)$ is an induced subgraph of $\mathcal{BG}_1$.

The following lemma is essentially proven in \cite[Lemma 6]{Lenz}.

\begin{lemma}\label{lemma:positivity_L(P)}
Let $(E,\chi,m)$ be a oriented arithmetic matroid with basis graph $\mathcal{BG}$, $B_0$ be a vertex of $\mathcal{BG}$ and $P$ be a coordinatizing path in a graph $\mathcal{G}$, such that $L(\mathcal{G})=\mathcal{BG}_1$.
Then there exists a re-orientation $\chi'$ of $\chi$ such that $\chi'(B)=\chi'(B_0)$ for all vertexes $B \in L(P)$.
\end{lemma}

We denote de point-wise product of two function $\chi$ and $m$ with 
\[\chi m(\underline{b}) \defeq \chi(\underline{b})\cdot m(\underline{b}).\]
We prove in our setting the equivalent of \cite[Lemma 9]{Lenz}.

\begin{lemma}\label{lemma:dist_1}
Let $(E,\rk, m)$ be an arithmetic matroid with basis graph $\mathcal{BG}$, $B_0$ be a vertex of $\mathcal{BG}$ and $P$ be a coordinatizing path in the graph $\mathcal{G}$, such that $L(\mathcal{G})=\mathcal{BG}_1$.
Let $\chi$ and $\chi'$ be two orientations of the arithmetic matroid $(E,\rk,m)$ such that $\chi(B)=\chi(B_0)$ and $\chi'(B)=\chi'(B_0)$ for all vertexes $B \in L(P)$.
If $\chi(B_0)=\chi'(B_0)$, then $\chi(B')=\chi'(B')$ for all $B' \in \mathcal{BG}_{\leq 1}$.
\end{lemma}
\begin{proof}
Consider the subgraph $\mathcal{H}$ of $\mathcal{G}$ with the same set of vertexes and with an edge between $i \in B_0$ and $j \in E \setminus B_0$ if and only if $\chi(B_0 \setminus \{ i \} \cup \{ j\} )=\chi'(B_0 \setminus \{ i \} \cup \{ j\})\neq 0$.
The graph $\mathcal{H}$ contains the chosen coordinatizing path $P$ by hypothesis.
Suppose that $\mathcal{H}\neq \mathcal{G}$ and let $T$ ($T\neq \emptyset$) be the set of edges of $\mathcal{G}$ not contained in $\mathcal{H}$.
For each $(i,j) \in T$ we can consider $l(i,j)$ the length of the minimal path in $\mathcal{H}$ connecting the vertexes $i$ and $j$.
Obviously, $l(i,j)$ is a odd number greater than $2$.
Let us fix $(h,k) \in T$ with $l(h,k)$ minimal among all $l(i,j)$ for $(i,j) \in T$ and a minimal path $Q=(h=i_0, j_0, i_1, \dots, i_t, j_t=k)$ in $\mathcal{H}$ between $(h,k)$ (the equality $2t+1=l(h,k)$ holds).
By minimality of $(h,k)$, two vertexes $i_a$ and $j_b$ are connected in $\mathcal{G}$ if and only if $a=b$, $a=b+1$ or $b=t$ and $a=0$.

Without loosing of generality, we suppose $i_v=v+1$ for $0\leq v \leq t$, $B_0=[r]$, and $j_v=r+v+1$ for $0\leq v \leq t$.
Apply \Cref{thm:Leibniz_formula} with $a_i=i$ and $b_j=t+j+2$ to the GP-functions $\chi m$ and $\chi' m$.
The product $\prod_{i=1}^r \chi m(a_1, \dots, b_{\sigma(i)}, \dots, a_r)$ is non zero if and only if $(a_i,b_{\sigma(i)})\in Q\cup \{(h,k)\}$ for all $i\leq t+1$ and $b_{\sigma(i)}=a_i$ for all $t+1 < i\leq r$. The same implication holds for the function $\chi' m$.
This happens only for two different permutations $\tau$ and $\eta$, say that $\tau(h)=k$ and $\tau(h)=j_0$.
We define 
\begin{align*}
x \defeq & \chi m(a_1, \dots, a_{h-1}, b_{k},a_{h+1}, \dots, a_r),\\
a \defeq & \prod_{i \neq h} \chi m(a_1, \dots, b_{\tau(i)}, \dots, a_r),\\
b \defeq & \prod_{i=1}^r \chi m(a_1, \dots, b_{\mu(i)}, \dots, a_r),\\
c \defeq & \chi m(a_1, \dots, a_r)^{r-1} \chi m(b_1, \dots, b_r).
\end{align*}
Thus, eq.~\eqref{eq:Leibniz_formula} can be reduced to $ ax+b=c$.
The equivalent relation for $\chi'$ is $ax'+b=c'$ with $x'=\pm x$ and $c'=\pm c$.
Since $a,b,c$ and $x$ are non-zero, then $x=x'$ and so 
\[\chi(a_1, \dots, a_{h-1}, b_{k},a_{h+1}, \dots, a_r)=\chi'(a_1, \dots, a_{h-1}, b_{k},a_{h+1}, \dots, a_r).\]
This equality contradicts the supposition $\mathcal{H}\neq \mathcal{G}$.
\end{proof}

\begin{lemma} \label{lemma:induction_basis_graph}
Let $(E,\rk,m)$ be an arithmetic matroid and $\chi$ and $\chi'$ two orientations of the arithmetic matroid $(E,\rk)$ that coincide on the elements of $\mathcal{BG}_{\leq 1}$.
Then $\chi = \chi'$.
\end{lemma}
\begin{proof}
By hypothesis both $\chi m$ and $\chi' m$ are GP-functions, so by \Cref{lemma:BG1_implies_BG} they are equal.
\end{proof}
\Cref{thm:uniqueness_orientation} follows from \Cref{lemma:positivity_L(P),lemma:dist_1,lemma:induction_basis_graph}.

\section{An example}
We show an example of orientable arithmetic matroid that is not representable.
\begin{example}
Let $([3], \rk, m)$ be the orientable arithmetic matroid associated with the matrix $\left( \begin{smallmatrix}
1 & 1 & 2\\
0 & n & n
\end{smallmatrix} \right) $.
Let $m'$ be the multiplicity function defined by $m'([3])=1$ and $m'(A)=m(A)$ for all $A\subsetneq [3]$.
The triple $([3], \rk, m')$ is a non-representable arithmetic matroid, since the multiplicity function does not have the GCD property.
This matroid is orientable, indeed any orientation $\chi$ of $([3],\rk, m)$ is an orientation of $([3],\rk, m')$.
\Cref{fig:pseudo_arr} represents an arrangement of hypersurfaces of $T^2$, the compact two dimensional torus, whose pattern of intersections coincides with the arithmetic matroid $([3],\rk,m')$ for $n=3$.
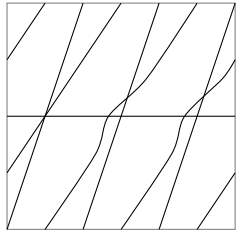
\begin{figure}
\centering
\begin{tikzpicture}[scale=0.5]
		\draw [gray] (0, 0) rectangle (6,6);
		\draw (0,3) -- (6,3);
		\draw (0,0) -- (2,6);
		\draw (2,0) -- (4,6);
		\draw (4,0) -- (6,6);
		\draw (0,1.5) -- (3,6);
		\draw (0,4.5) -- (1,6);
		\draw (5,0) -- (6,1.5);
		\draw plot [smooth] coordinates {(3,0) (4.33,2) (4.67,3) (5.67,4) (6,4.5)};
		\draw plot [smooth] coordinates {(1,0) (2.33,2) (2.67,3) (3.67,4) (5,6)};
\end{tikzpicture}
\caption{An arrangement of hypersurfaces in the compact torus.} \label{fig:pseudo_arr}
\end{figure}
\end{example}
\section{Representability}

\begin{definition}\label{def:strongGCD}
An arithmetic matroid $(E,\rk,m)$ has the \textit{strong GCD property} if 
\[ m(A)= \gcd \{ m(B) \mid B \textnormal{ basis and } |B\cap A|= \rk A\}\]
for all $A\subseteq E$.
\end{definition}
\noindent
Notice that arithmetic matroids with the strong GCD property are uniquely determine by their values on the basis of the underlying matroid.
The strong GCD property is equivalent to both $(E,\rk,m)$ and $(E,\rk^*,m^*)$ are GCD arithmetic matroids.

\begin{prop}\label{prop:repr_matroid}
Let $(E,\rk,m)$ be an orientable arithmetic matroid. Then the underlying matroid $(E,\rk)$ is representable over $\Q$.
\end{prop}
\begin{proof}
We choose an orientation $\chi$ of the arithmetic matroid $(E,\rk,m)$ and a basis $B_0=(b_1, \dots b_r)$ of the matroid.
For each $e \in E$, consider in $\Q^r$ the vector 
\[v_e \defeq (\chi m(b_1, \dots, b_{i-1}, e, b_{i+1}, \dots, b_r) 
)_{1\leq i \leq r}.\]
We choose a total order on $E= [n]$ such that $B_0=[r]$.
Let $N$ be the matrix that represent the vectors $v_i$, for $i=1,\dots, n$, in the canonical basis of $\Q^r$.
We claim that, for each $A\subseteq [n]$ of cardinality $r$, the functions $\det N[A]$ and $\chi m(B_0)^{r-1}\chi m (A)$ coincide.
The claimed equality holds if $A=B_0$. 
If $A=\{1,\dots, i-1,i+1, \dots, r,j\}$, then 
\begin{align*}
\det N[A]&=(-1)^{r-i}\frac{\chi m(1,\dots, i-1,j,i+1,\dots,r)}{\chi m([r])}\det N[[r]]\\
&=\frac{\chi m(A)}{\chi m(B_0)} \chi m (B_0)^r =\chi m(B_0)^{r-1}\chi m (A)
\end{align*}

The GP-function $\chi m(B_0)^{r-1} \chi m(\cdot)$ and $\det N[\cdot]$ coincide on $\mathcal{BG}_{\leq 1}$, thus by \Cref{lemma:BG1_implies_BG} $\chi m(B_0)^{r-1} \chi m(B)=\det N[B]$ for all $B\subset E$, $|B|=r$.
The matroid defined by $N$ is $(E,\rk)$ since they have the same set of basis.

\end{proof}

\begin{prop}\label{prop:strongGCD+orient_realiz}
Let $(E,\rk,m)$ be an orientable arithmetic matroid with the strong GCD property. Then $(E,\rk,m)$ is representable.
\end{prop}
\begin{proof}
Consider a orientation of $(E,\rk,m)$, the vectors $v_e \in \Q^r$ for $e \in E$ defined in the proof of \Cref{prop:repr_matroid}, and let $\Lambda$ the lattice generated by $\{v_e\}_{e\in E}$.
Let $G$ be a finite abelian group of cardinality $m(\emptyset)=m(E)$.
We claim that the elements $(v_e,0)$ in $\Lambda \times G$ are a representation of the arithmetic matroid $(E,\rk,m)$.
Observe that the index $[\Z^r: \Lambda]$ is equal to $m(B_0)^{r-1} m(E)$.
Let $(E,\rk,m')$ be the arithmetic matroid described by the vectors $(v_e,0)$.
The multiplicity functions $m$ and $m'$ coincides on all basis of the matroid $(E,rk)$, hence by the GCD property $m=m'$.
\end{proof}

\bibliographystyle{alpha}
\bibliography{Orient_arith_matroids}

\end{document}